\documentclass[a4paper,12pt,reqno]{amsart}
\usepackage{a4wide}
\usepackage{amsmath}
\usepackage{amssymb}
\usepackage{amsthm}
\usepackage{latexsym}
\usepackage{graphicx}
\usepackage[english]{babel}
              
\newtheorem{prop}[subsection]{Proposition}

\newtheorem{teor}[subsection]{Theorem}
\newtheorem{lema}[subsection]{Lemma}
\newtheorem{cor} [subsection]{Corollary}
\theoremstyle{definition}

\theoremstyle{remark}
\newtheorem{obs} [subsection]{Remark}
\newtheorem{exm} [subsection]{Example}

\newcommand{\Zng}{$\mathbb Z^n$-graded $S$-module}

\def\sdepth{\operatorname{sdepth}}
\def\qdepth{\operatorname{hdepth}}
\def\depth{\operatorname{depth}}

\def\deg{\operatorname{deg}}

\def\PP{\operatorname{P}}

\numberwithin{equation}{section}

\begin{document}

\title[Comb.inequalities related to squarefree monomial ideals]{Several combinatorial inequalities related to squarefree monomial ideals}
\author[Silviu B\u al\u anescu, Mircea Cimpoea\c s 
       ]
  {Silviu B\u al\u anescu$^1$ and Mircea Cimpoea\c s$^2$
	}
\date{}



\footnotetext[1]{ \emph{Silviu B\u al\u anescu}, University Politehnica of Bucharest, Faculty of
Applied Sciences, 
Bucharest, 060042, E-mail: silviu.balanescu@stud.fsa.upb.ro}
\footnotetext[2]{ \emph{Mircea Cimpoea\c s}, University Politehnica of Bucharest, Faculty of
Applied Sciences, 
Bucharest, 060042, Romania and Simion Stoilow Institute of Mathematics, Research unit 5, P.O.Box 1-764,
Bucharest 014700, Romania, E-mail: mircea.cimpoeas@upb.ro,\;mircea.cimpoeas@imar.ro}

\begin{abstract}
Let $K$ be a field and $S=K[x_1,\ldots,x_n]$, the ring of polynomials in $n$ variables, over $K$. Using the fact that the 
Hilbert depth is an upper bound for the Stanley depth of a quotient of squarefree
monomial ideals $0\subset I\subsetneq J\subset S$, we prove several combinatorial inequalities which involve the coefficients of the polynomial $f(t)=(1+t+\cdots+t^{m-1})^n$.
\end{abstract}

\maketitle

\section{Introduction}

Let $K$ be a field and $S=K[x_1,\ldots,x_n]$ the polynomial ring over $K$.
Let $M$ be a \Zng. A \emph{Stanley decomposition} of $M$ is a direct sum $\mathcal D: M = \bigoplus_{i=1}^rm_i K[Z_i]$ as a 
$\mathbb Z^n$-graded $K$-vector space, where $m_i\in M$ is homogeneous with respect to $\mathbb Z^n$-grading, 
$Z_i\subset\{x_1,\ldots,x_n\}$ such that $m_i K[Z_i] = \{um_i:\; u\in K[Z_i] \}\subset M$ is a free $K[Z_i]$-submodule of $M$. 
We define $\sdepth(\mathcal D)=\min_{i=1,\ldots,r} |Z_i|$ and $$\sdepth(M)=\max\{\sdepth(\mathcal D)|\;\mathcal D\text{ is 
a Stanley decomposition of }M\}.$$ The number $\sdepth(M)$ is called the \emph{Stanley depth} of $M$. 

Herzog, Vladoiu and Zheng show in \cite{hvz} that $\sdepth(M)$ can be computed in a finite number of steps if $M=I/J$, 
where $J\subset I\subset S$ are monomial ideals. 
In \cite{apel}, J.\ Apel restated a conjecture firstly given by Stanley in 
\cite{stan}, namely that $$\sdepth(M)\geq\depth(M),$$ for any \Zng $\;M$. This conjecture proves to be false, in general, for 
$M=S/I$ and $M=J/I$, where $0\neq I\subset J\subset S$ are monomial ideals, see \cite{duval}, but remains open for $M=I$.
Another open question, proposed by Herzog \cite{her}, is the following: Is it true that $\sdepth(I)\geq \sdepth(S/I)+1$ for any monomial ideal $I$?

The explicit computation of the Stanley depth is a computational difficult task, even in the very special case of the ideal $\mathbf m=(x_1,\ldots,x_n)$,
see \cite{biro}. This is one of the reason, a new invariant, associated to (multi)graded $S$-modules, called Hilbert depth, was introduced,
which gives a natural upper bound for the Stanley depth. See \cite{bruns} for further details. 

In \cite{lucrare2}, we proved a new
formula for the Hilbert depth of $J/I$, where $0\neq I\subsetneq J\subset S$ are squarefree monomial ideals; see Section $2$.

In this paper, our aim is to deduce several combinatorial inequalities, using the fact that $\qdepth(J/I)\geq \sdepth(J/I)$, in
certain cases when sharp lower bounds for $\sdepth(J/I)$ are known. 
In Section $3$, we consider the path ideal of length $m$ of a path graph, i.e. 
$$I_{n,m}=(x_1x_2\cdots x_m,\;x_{2}x_3\cdots x_{m+1},\;\ldots,x_{n-m+1}\cdots x_n)\subset S.$$
In Proposition \ref{aknm} we show that the number of squarefree monomials of degree $k$ which do not belong to $I_{n,m}$
is equal to 
$$\binom{n-k+1,\;m}{k}:=\text{ the coefficient of }x^k\text{ in the expansion of }(1+x+\cdots+x^{m-1})^{n-k+1}.$$
We mention that the above expansion was firstly studied by Euler \cite{euler}. For a modern perspective, we refer the
reader to \cite[Page 77]{comtet}.

Using this, a formula for $\sdepth(S/I_{n,m})$ and the fact that $\sdepth(I_{n,m})\geq \depth(I_{n,m})$, in Theorem \ref{t31} we prove that for
$n\geq m\geq 1$ and $d:=\varphi(n,m)=n+1-\left\lfloor \frac{n+1}{m+1} \right\rfloor-\left\lceil \frac{n+1}{m+1} \right\rceil$, we have that:
\begin{enumerate}
\item[(1)] $\sum\limits_{j=0}^k (-1)^{k-j} \binom{d-j}{k-j}\binom{n-j+1,\; m}{j} = \sum\limits_{\ell=0}^{\left\lfloor \frac{k}{m} \right\rfloor} (-1)^{\ell}
\sum\limits_{j=m\ell}^k (-1)^{k-j} \binom{d-j}{k-j}\binom{n-j+1}{\ell}\binom{n-m\ell}{j-m\ell}$, $0\leq k\leq d$.
\item[(2)] $ \sum\limits_{j=0}^k (-1)^{k-j} \binom{d-j}{k-j}\binom{n-j+1,\; m}{j} \geq 0$, for all $0\leq k\leq d$.
\item[(3)] $ \sum\limits_{j=0}^k (-1)^{k-j} \binom{d+1-j}{k-j} \binom{n-j+1,\; m}{j} \leq \binom{n-d+k-2}{k} $, for all $0\leq k\leq d+1$.
\end{enumerate}
Also, we particularize (2-3) for $m=2$; see Corollary \ref{caz2-1}.

In Section $4$, we consider the path ideal of length $m$ of a cycle graph, i.e. 
$$J_{n,m}=I_{n,m}+(x_{n-m+2}\cdots x_nx_1,\ldots,x_nx_1\cdots x_{m-1})\subset S,$$
where $n>m\geq 2$ are some integers. Let $d=n - \left\lfloor \frac{n}{m+1} \right\rfloor - \left\lceil \frac{n}{m+1} \right\rceil$.

In Proposition \ref{ppp} we prove that $\sdepth(J_{n,m})\geq \depth(J_{n,m})$, i.e. $J_{n,m}$ satisfies the Stanley inequality, and, 
also, if $m\geq 3$ then $\sdepth(J_{n,m})\geq \sdepth(S/J_{n,m})+1$.
Using Proposition \ref{ppp} and some other results about $\sdepth(S/J_{n,m})$ and $\sdepth(J_{n,m}/I_{n,m})$, 
in Theorem \ref{t33} we prove that:
\begin{enumerate}
\item[(1)] $\sum\limits_{j=0}^k (-1)^{k-j} \binom{d-j}{k-j}\left( \binom{n-j+1,\; m}{j} - 
\sum\limits_{\ell=m}^{2m-2} (2m-1-\ell) \binom{n-\ell-j+1,\;m}{j-\ell}  \right) \geq 0$ for all $0\leq k\leq d$.
\item[(2)] $\sum\limits_{j=0}^k (-1)^{k-j} \binom{d+m-1-j}{k-j}
\sum\limits_{\ell=m}^{2m-2} (2m-1-\ell) \binom{n-\ell-j+1,\;m}{j-\ell}  \geq 0$ for all $0\leq k\leq d+m-1$.
\item[(3)] $\sum\limits_{j=0}^k (-1)^{k-j} \binom{d+1-j}{k-j}
\left( \binom{n-j+1,\;m}{j} - \sum\limits_{\ell=m}^{2m-2} (2m-1-\ell) \binom{n-\ell-j+1,\;m}{j-\ell} \right) \leq 
\binom{n-d+k-2}{k}$, 

for all $0\leq k\leq d+1$. In Corollary \ref{caz2-2} we particularize (1-3) for $m=2$.
\end{enumerate}

\section{Preliminaries}

First, we fix some notations and we recall the main result of \cite{lucrare2}. 

We denote $[n]:=\{1,2,\ldots,n\}$ and $S:=K[x_1,\ldots,x_n]$. 

For a subset $C\subset [n]$, we denote $x_C:=\prod_{j\in C}x_j\in S$.

For two subsets $C\subset D\subset [n]$, we denote $[C,D]:=\{A\subset [n]\;:\;C\subset A\subset D\}$,
      and we call it the \emph{interval} bounded by $C$ and $D$.

Let $I\subset J\subset S$ be two square free monomial ideals. We let:
$$\PP_{J/I}:=\{C\subset [n]\;:\;x_C\in J\setminus I\} \subset 2^{[n]}.$$
A partition of $P_{J/I}$ is a decomposition: $$\mathcal P:\;\PP_{J/I}=\bigcup_{i=1}^r [C_i,D_i],$$
into disjoint intervals.

If $\mathcal P$ is a partition of $\PP_{J/I}$, we let $\sdepth(\mathcal P):=\min_{i=1}^r |D_i|$.
The Stanley depth of $P_{J/I}$ is 
      $$\sdepth(P_{J/I}):=\max\{\sdepth(\mathcal P)\;:\;\mathcal P\text{ is a partition of }\PP_{J/I}\}.$$
Herzog, Vl\u adoiu and Zheng proved in \cite{hvz} that: $$\sdepth(J/I)=\sdepth(\PP_{J/I}).$$
Let $\PP:=\PP_{J/I}$, where $I\subset J\subset S$ are square-free monomial ideals. For any $0\leq k\leq n$, we denote:
$$\PP_k:=\{A\in \PP\;:\;|A|=k\}\text{ and }\alpha_k(J/I)=\alpha_k(\PP)=|\PP_k|.$$
For all $0\leq d\leq n$ and $0\leq k\leq d$, we consider the integers
\begin{equation}\label{betak}
  \beta_k^d(J/I):=\sum_{j=0}^k (-1)^{k-j} \binom{d-j}{k-j} \alpha_j(J/I).
\end{equation}
From \eqref{betak} we can easily deduce that
\begin{equation}\label{alfak}
  \alpha_k(J/I)=\sum_{j=0}^k \binom{d-j}{k-j} \beta_k^d(J/I),\text{ for all }0\leq k\leq d.
\end{equation}
Also, we have that
\begin{equation}\label{betak2}
\beta_k^d(J/I) = \alpha_k(J/I) - \binom{d}{k}\beta_0^d(J/I)-\binom{d-1}{k-1}\beta_1^d(J/I)-\cdots-\binom{d-k+1}{1}\beta_{k-1}^d(J/I).
\end{equation}

\begin{teor}(\cite[Theorem 2.4]{lucrare2})\label{d1}
With the above notations, the \emph{Hilbert depth} of $J/I$ is
$$\qdepth(J/I):=\max\{d\;:\;\beta_k^d(J/I) \geq 0\text{ for all }0\leq k\leq d\}.$$
\end{teor}

As a basic property of the Hilbert depth, we state the following:

\begin{prop}\label{p1}
Let $I\subset J\subset S$ be two square-free monomial ideals. Then $$\sdepth(J/I)\leq \qdepth(J/I).$$
\end{prop}

\section{The $m$-path ideal of a path graph}

Let $n\geq m\geq 1$ be two integers and 
$$I_{n,m}=(x_1x_2\cdots x_m,\;x_{2}x_3\cdots x_{m+1},\;\ldots,x_{n-m+1}\cdots x_n)\subset S,$$
be the $m$-path ideal associated to the $m$-path of length $n$. We define:
\begin{equation}\label{e31}
\varphi(n,m):=n+1 - \left\lfloor \frac{n+1}{m+1} \right\rfloor - \left\lceil \frac{n+1}{m+1} \right\rceil.
\end{equation}
According to \cite[Theorem 1.3]{path} we have that 
\begin{equation}\label{e32}
\sdepth(S/I_{n,m}) = \depth(S/I_{n,m})=\varphi(n,m).
\end{equation}
Also, according to \cite[Proposition 1.7]{comm}, we have that
\begin{equation}\label{e33}
\sdepth(I_{n,m}) \geq \depth(I_{n,m})=\varphi(n,m)+1.
\end{equation}

Note that $\alpha_k(S/I_{n,m})$ counts the number of squarefree monomial of degree $k$ which do not belong to $I_{n,m}$.

\begin{prop}\label{aknm}
With the above notations, $\alpha_k(S/I_{n,m})=\binom{n-k+1,\;m}{k}:=$ the coefficient of $x^k$ 
from the expansion $(1+x+\cdots+x^{m-1})^{n-k+1}$.

In particular $\alpha_k(I_{n,m})=\binom{n}{k}-\binom{n-k+1,\;m}{k}$.
\end{prop}

\begin{proof}
Note that the coefficient of $x^k$ from the expansion $(1+x+\cdots+x^{m-1})^{n-k+1}$ is equal
to the number of sequences $(a_1,a_2,\ldots,a_{n-k+1})$ with $a_i\in\{0,1,\ldots,m-1\}$ for all $1\leq i\leq n-k+1$,
such that $a_1+a_2+\cdots+a_{n-k+1}=k$. Therefore, in order to complete the proof,
it is enough to establish a 1-to-1 correspondence between the squarefree monomials $u\in S\setminus I_{n,m}$ of degree $k$.

Indeed, given a sequence $\mathbf a=(a_1,a_2,\ldots,a_{n-k+1})$ as above, we define a monomial $u_{\mathbf a}$ as follows:
We let $A_j:=\sum_{i=1}^j a_i$ for all $1\leq j\leq n-k$, and $A_0=0$. We define
$$u_{\mathbf a}=\prod_{i=1}^{n-k+1} u_i,\;\text{ where }u_i=\begin{cases} 1,& a_i=0 \\ x_{i+A_i}x_{i+1+A_i}\cdots x_{i+a_i-1+A_i},&a_i>0 \end{cases}.$$
It is clear that $\deg(u_{\mathbf a})=k$ and $u\notin I_{n,m}$, since, by construction, there is no monomial of the
form $x_ix_{i+1}\cdots x_{i+m-1}$ which divides $u_{\mathbf a}$.

Conversely, let $u\in S\setminus I_{n,m}$ be a squarefree monomial of degree $k$. We can write $u$ as a product $u=w_1w_2\cdots w_t$, where
$w_j = x_{i_j}x_{i_j+1}\cdots x_{i_j+b_j-1}$ for all $1\leq j\leq t$ such that:
$$1\leq i_1 < i_2 <\cdots <i_t \leq n-b_t+1,\;b_1+\cdots+b_t=k\text{ and }i_j+b_j < i_{j+1}\text{ for all }1\leq j\leq t-1.$$
Note that, since $u\notin I_{n,m}$, we have that $b_j=\deg(w_j)\leq m-1$ for all $1\leq j\leq t$.

We construct a sequence $\mathbf a=(a_1,\ldots,a_{n-k+1})$ as follows:

We let $a_{i_j-j+1}=b_j$ for all $1 \leq j\leq t$ and $a_i=0$ whenever $i\neq i_j-j+1$ for all $1\leq j\leq t$. It is easy
to see that $\sum_{i=1}^{n-k+1}a_i=k$ and $a_i\in\{0,1,\ldots,m-1\}$ for all $1\leq i\leq n-k+1$. 
Moreover, we have that $u=u_{\mathbf a}$. Hence, the proof is complete.
\end{proof}

\begin{exm}\rm
Let $n=7$, $m=3$ and $k=4$. According to Proposition \ref{aknm}, $\alpha_4(S/I_{7,3})=$ the
coefficient of $x^4$ in the expansion of $(1+x+x^2)^4$. By straightforward computations, we get $\alpha_4(S/I_{7,3})=19$.
Let $\mathbf a=(0,1,1,2)$ be a sequence,
as in the proof of Proposition \ref{aknm}. The corresponding monomial is $u=x_2x_4x_6x_7\in S\setminus I_{7,3}$.
Similarly, if $u'=x_1x_3x_4x_5\in S\setminus I_{7,3}$, then $u'=u_{\mathbf a'}$, where $\mathbf a'=(1,3,0,0)$.
\end{exm}

\begin{obs}\rm
Let $n\geq m\geq 1$ and $m\leq k\leq n$ be some integers such that $n\geq 2m$. 
Let $L_0:=I_{n,m}$ and $L_i:=L_{i-1}:x_{n-i+1}$ for $1\leq i\leq m-1$.
We consider the short exact sequences:
\begin{equation}\label{re1}
0\to S/L_i \stackrel{\cdot x_{n-i+1}}{\rightarrow} S/L_{i-1} \to S/(L_{i-1},x_{n-i+1}) \to 0\text{ for }1\leq i\leq m-1.
\end{equation}
We denote $S_j:=K[x_1,\ldots,x_j]$ for any $1\leq j\leq n$. We have that 
\begin{equation}\label{re2}
S/(L_{i-1},x_{n-i+1})\cong S/(I_{n-i,m}S,x_{n-i+1}) = (S_{n-i}/I_{n-i,m})[x_{n-i+2},\ldots,x_n]\text{ for }1\leq i\leq m-1,
\end{equation}
Also $L_{m-1}=(I_{n-m,m},x_{n-m+1})$ and therefore
\begin{equation}\label{re3}
S/L_{m-1}\cong (S_{n-m}/I_{n-m,m})[x_{n-m+2},\ldots,x_n].
\end{equation}
From \eqref{re1}, \eqref{re2} and \eqref{re3} it follows that
\begin{equation}\label{re4}
\alpha_k(S/I_{n,m})=\alpha_{k}(S_{n-1}/I_{n-1,m})+\cdots+\alpha_{k-m+1}(S_{n-m}/I_{n-m,m}).
\end{equation}
Let $N:=n-k+1$. From Proposition \ref{aknm} and \eqref{re4} we reobtain the identity:
$$\binom{N,\;m}{k} = \binom{N-1,\;m}{k} + \binom{N-1,\;m}{k-1}+ \cdots + \binom{N-1,\;m}{k-m+1}.$$
\end{obs}

\begin{lema}(See also \cite[Page 77]{comtet})\label{kruk}
Let $n\geq m\geq 1$ and $0\leq k\leq n$ be some integers. Then:
$$ \binom{n-k+1,\;m}{k}= \sum_{\ell=0}^{\left\lfloor \frac{k}{m} \right\rfloor} (-1)^{\ell} \binom{n-k+1}{\ell}\binom{n-m\ell}{k-m\ell}.$$
\end{lema}

\begin{proof}
We have that
\begin{equation}\label{cruc}
 (1+t+\cdots+t^{m-1})^{n-k+1} = \left(\frac{1-t^m}{1-t}\right)^{n-k+1} = \sum_{\ell=0}^{n-k+1}(-1)^{\ell}\binom{n-k+1}{\ell}t^{m\ell}
  \sum_{j=0}^{\infty}\binom{n+j-1}{j}t^j.
\end{equation}
Identifying the coefficient of $t^k$ in \eqref{cruc} we deduce the required conclusion.
\end{proof}

We also recall the following combinatorial identity
\begin{equation}\label{chuv}
\sum_{j=0}^k (-1)^{k-j}\binom{d-j}{k-j}\binom{n}{j} = \binom{n-d+k-1}{k},
\end{equation}
which is a direct consequence of the Chu-Vandermonde summation.

\begin{teor}\label{t31}
Let $n\geq m\geq 1$ and $d:=\varphi(n,m)=n+1-\left\lfloor \frac{n+1}{m+1} \right\rfloor-\left\lceil \frac{n+1}{m+1} \right\rceil$. We have that:
\begin{enumerate}
\item[(1)] $\sum\limits_{j=0}^k (-1)^{k-j} \binom{d-j}{k-j}\binom{n-j+1,\; m}{j} = \sum\limits_{\ell=0}^{\left\lfloor \frac{k}{m} \right\rfloor} (-1)^{\ell}
\sum\limits_{j=m\ell}^k (-1)^{k-j} \binom{d-j}{k-j}\binom{n-j+1}{\ell}\binom{n-m\ell}{j-m\ell}$, $0\leq k\leq d$.
\item[(2)] $ \sum\limits_{j=0}^k (-1)^{k-j} \binom{d-j}{k-j}\binom{n-j+1,\; m}{j} \geq 0$, for all $0\leq k\leq d$.
\item[(3)] $ \sum\limits_{j=0}^k (-1)^{k-j} \binom{d+1-j}{k-j} \binom{n-j+1,\; m}{j} \leq \binom{n-d+k-2}{k} $, for all $0\leq k\leq d+1$.
\end{enumerate}
\end{teor}

\begin{proof}
(1) It follows immediately from Lemma \ref{kruk}.

(2) From Proposition \ref{p1}, Proposition \ref{aknm}, \eqref{betak} and \eqref{e32} it follows that 
    $$ \beta_k^d(S/I_{n,m})=\sum_{j=0}^k (-1)^{k-j} \binom{d-j}{k-j}\binom{n-j+1,\;m}{j} \geq 0,$$
    for all $0\leq k\leq d$, as required.

(3) It follows from Proposition \ref{p1}, Proposition \ref{aknm}, \eqref{betak}, \eqref{chuv} and \eqref{e33}.
\end{proof}

\begin{cor}\label{caz2-1}(Case $m=2$)
Let $n\geq 2$ be an integer. We have that
\begin{enumerate}

\item[(1)] $\sum\limits_{j=0}^k (-1)^{k-j} \binom{d-j}{k-j}\binom{n-j+1,\; m}{j} = \sum\limits_{\ell=0}^{\left\lfloor \frac{k}{m} \right\rfloor} (-1)^{\ell}
\sum\limits_{j=m\ell}^k (-1)^{k-j} \binom{d-j}{k-j}\binom{n-j+1}{\ell}\binom{n-m\ell}{j-m\ell}$, $0\leq k\leq d$.

\item[(1)] $\sum\limits_{j=0}^k (-1)^{k-j} \binom{\left\lceil \frac{n}{3} \right\rceil-j}{k-j}\binom{n-j+1}{j} \geq 0,\text{ for all }0\leq k\leq \left\lceil \frac{n}{3} \right\rceil.$
\item[(2)] $\sum\limits_{j=0}^k (-1)^{k-j} \binom{\left\lceil \frac{n}{3} \right\rceil+1-j}{k-j}\binom{n-j+1}{j}
\leq \binom{ \left\lfloor \frac{2n}{3} \right\rfloor + k - 2}{k} 
,\text{ for all }0\leq k\leq \left\lceil \frac{n}{3} \right\rceil+1.$
\end{enumerate}
\end{cor}

\begin{proof}
(1) From Proposition \ref{aknm}, it follows that $\alpha_k(S/I_{n,2})=\binom{n-k+1}{k}$ for all $0\leq k\leq n$.
The conclusion follows from Theorem \ref{t31}.
\end{proof}


\section{The $m$-path ideal of a cycle graph}

Let $n > m\geq 2$ be two integer and
$$J_{n,m}=I_{n,m}+(x_{n-m+2}\cdots x_nx_1,\ldots,x_nx_1\cdots x_{m-1})\subset S,$$
the $m$-path ideal associated to the cycle graph of length $n$.
According to \cite[Theorem 1.4]{ciclu} we have that 
\begin{equation}\label{eq31}
 \varphi(n,m)\geq \sdepth(S/J_{n,m}) \geq \depth(S/J_{n,m})=\varphi(n-1,m).
\end{equation}
Also, according to \cite[Proposition 1.6]{ciclu} we have that 
\begin{equation}\label{eq32}
\sdepth(J_{n,m}/I_{n,m}) \geq \depth(J_{n,m}/I_{n,m})\geq \varphi(n-1,m)+m-1.
\end{equation}

\begin{prop}\label{ppp}
Let $n>m\geq 2$ be two integers. We have that:
\begin{enumerate}
\item[(1)] $\sdepth(J_{n,m})\geq \min\{ \varphi(n-1,m)+m-1,\; \varphi(n,m)+1 \}$.
\item[(2)] $\sdepth(J_{n,m})\geq \depth(J_{n,m}) = \varphi(n-1,m)+1$. 
\item[(3)] If $m\geq 3$ then $\sdepth(J_{n,m})\geq \varphi(n,m)+1$, so
            $\sdepth(J_{n,m})\geq \sdepth(S/J_{n,m})+1$.   
\end{enumerate}
\end{prop}

\begin{proof}
(1) We consider the short exact sequence $0\to I_{n,m} \to J_{n,m} \to \frac{J_{n,m}}{I_{n,m}} \to 0.$
    From \cite[Lemma 2.2]{asia}, \eqref{e33} and \eqref{eq32} we get the required result.
		
(2) It follows from (1) and \eqref{eq31}, since $\varphi(n,m)\geq \varphi(n-1,m)$ and $m\geq 2$.

(3) Since $m\geq 3$, it follows that $\varphi(n-1,m)+m-1\geq \varphi(n-1,m)+2\geq \varphi(n,m)+1$,
    hence, the first inequality follows from (1). The second inequality follows from \eqref{eq31}.
\end{proof}

\begin{prop}\label{jpei}
With the above notations, 
$$\alpha_k(J_{n,m}/I_{n,m}) = \sum_{\ell=m}^{2m-2} (2m-1-\ell) \binom{n-\ell-k+1,\;m}{k-\ell}.$$
\end{prop}

\begin{proof}
As in the proof of Proposition \ref{aknm}, we put squarefree monomials of degree $k$ which are not
in $I_{n,m}$ in bijection with sequences of the form $\mathbf a=(a_1,\ldots,a_{n-k+1})$, where $a_i$'s are 
integers such that $0\leq a_i\leq m-1$ for all $i$ and $\sum_{i=1}^{n-k+1}a_i=k$. 
Let $u\in J_{n,m}\setminus I_{n,m}$ such that $\deg(u)=k$. Assume $u=u_{\mathbf a}$ for a sequence $\mathbf a$
as above. Since $u\in J_{n,m}$ it follows that $\ell:=a_1+a_{n-k+1}\geq m$. On the other hand, $\ell\leq 2(m-1)=2m-2$.
Also, for a given $\ell \in \{m,m+1,\ldots,2m-2\}$ there are exactly $(2m-1-\ell)$ pairs $(a_1,a_{n-k+1})$ such that
$a_1+a_{n-k+1}=\ell$ and $0\leq a_1,a_{n-k+1} \leq m-1$. Note that the sequence $(a_2,a_3,\ldots,a_{n-k})$ has length $n-k-1$
and satisfy the conditions $0 \leq a_i \leq m-1$, for all $2\leq i\leq n-k$, and $\sum_{i=2}^{n-k}=k-\ell$. 
Therefore, there
are $\binom{n-\ell-k+1,\;m}{k-\ell}$ ways in which we can choose such sequences.
Since the monomial $u$ is uniquely determined by the pair $(a_0,a_{n-k+1})$ and the sequence $(a_2,a_3,\ldots,a_{n-k})$ as above,
we get the required conclusion.
\end{proof}

\begin{prop}\label{p32}
Let $n > m\geq 2$ and $0\leq k\leq n$ be some integers. We have that
\begin{enumerate}
\item[(1)] $\alpha_k(S/J_{n,m})=\binom{n-k+1,\;m}{k} -\sum\limits_{\ell=m}^{2m-2} (2m-1-\ell) \binom{n-\ell-k+1,\;m}{k-\ell}$.
\item[(2)] $\alpha_k(J_{n,m})=\binom{n}{k} - \binom{n-k+1,\;m}{k} + \sum\limits_{\ell=m}^{2m-2} (2m-1-\ell) \binom{n-\ell-k+1,\;m}{k-\ell}$.
\end{enumerate}
\end{prop}

\begin{proof}
(1) It follows from Proposition \ref{jpei}, Proposition \ref{aknm} and the obvious fact that
$$\alpha_k(S/J_{n,m})=\alpha_k(S/I_{n,m})-\alpha_k(J_{n,m}/I_{n,m}).$$
(2) If follows from (1) and the fact that $\alpha_k(J_{n,m})=\binom{n}{k}-\alpha_k(S/J_{n,m})$.
\end{proof}

\begin{teor}\label{t33}
Let $n> m \geq 2$ be some integers. Let $d=n - \left\lfloor \frac{n}{m+1} \right\rfloor - \left\lceil \frac{n}{m+1} \right\rceil$.
We have
\begin{enumerate}
\item[(1)] $\sum\limits_{j=0}^k (-1)^{k-j} \binom{d-j}{k-j}\left( \binom{n-j+1,\; m}{j} - 
\sum\limits_{\ell=m}^{2m-2} (2m-1-\ell) \binom{n-\ell-j+1,\;m}{j-\ell}  \right) \geq 0$ for $0\leq k\leq d$.
\item[(2)] $\sum\limits_{j=0}^k (-1)^{k-j} \binom{d+m-1-j}{k-j}
\sum\limits_{\ell=m}^{2m-2} (2m-1-\ell) \binom{n-\ell-j+1,\;m}{j-\ell}  \geq 0$ for all $0\leq k\leq d+m-1$.
\item[(3)] $\sum\limits_{j=0}^k (-1)^{k-j} \binom{d+1-j}{k-j}
\left( \binom{n-j+1,\;m}{j} - \sum\limits_{\ell=m}^{2m-2} (2m-1-\ell) \binom{n-\ell-j+1,\;m}{j-\ell} \right) \leq 
\binom{n-d+k-2}{k}$, 

\end{enumerate}
\end{teor}

\begin{proof}
(1) As in the proof of Theorem \ref{t31}, the conclusion follows from 
    Proposition \ref{p1}, Proposition \ref{p32}(1), \eqref{e31}, \eqref{eq31} and \eqref{betak}.\\
		(2) It follows from Proposition \ref{p1}, Proposition \ref{jpei}, \eqref{e31}, \eqref{eq32} and \eqref{betak}.

\noindent
(3) It follows from Proposition \ref{p1}, Proposition \ref{p32}(2), Proposition \ref{ppp}(2), \eqref{betak} and \eqref{chuv}.
\end{proof}

\begin{cor}\label{caz2-2}(Case $m=2$)
Let $n\geq 3$ be an integer. We have that:
\begin{enumerate}
\item[(1)] $\sum\limits_{j=0}^k (-1)^{k-j} \binom{\left\lceil \frac{n-1}{3} \right\rceil -j}{k-j}\binom{n-j+1}{j}
        \geq \sum\limits_{j=2}^k (-1)^{k-j} \binom{\left\lceil \frac{n-1}{3} \right\rceil -j}{k-j}\binom{n-j+1}{j-2}$ for all $0\leq k\leq \left\lceil \frac{n-1}{3} \right\rceil$.
\item[(2)] $\sum\limits_{j=2}^k (-1)^{k-j} \binom{\left\lceil \frac{n+2}{3} \right\rceil -j}{k-j}\binom{n-j+1}{j-2}\geq 0$ for all 
$2\leq k\leq \left\lceil \frac{n+2}{3} \right\rceil$.
\item[(3)] $\binom{\left\lfloor \frac{2n+1}{3} \right\rfloor+k-2}{k} - 
           \sum\limits_{j=0}^k (-1)^{k-j} \binom{\left\lceil \frac{n+2}{3} \right\rceil -j}{k-j} \binom{n-j+1}{j}
        \geq \sum\limits_{j=2}^k (-1)^{k-j} \binom{\left\lceil \frac{n+2}{3} \right\rceil -j}{k-j}\binom{n-j+1}{j-2}$ for all $0\leq k\leq \left\lceil \frac{n+2}{3} \right\rceil$.
\end{enumerate}
\end{cor}

\begin{proof}
It follows immediately from Theorem \ref{t33}.
\end{proof}

 \subsection*{Acknowledgements}

The second author was supported by a grant of the Ministry of Research, Innovation and Digitization, CNCS - UEFISCDI, 
project number PN-III-P1-1.1-TE-2021-1633, within PNCDI III.

\end{document}